\documentclass[fleqn,11pt,a4paper]{amsart}

\usepackage{amsmath, amssymb, enumerate}

\usepackage{xy} \xyoption{all}

\usepackage{hyperref}

\numberwithin{equation}{section}

\def\today{\number\day\space\ifcase\month\or   January\or February\or
   March\or April\or May\or June\or   July\or August\or September\or
   October\or November\or December\fi\   \number\year}

\newtheorem{lma}{Lemma}

\newtheorem{thm}{Theorem}

\newtheorem{cor}{Corollary}

\newtheorem{prp}{Proposition}

\theoremstyle{definition}
\newtheorem{pgr}{}
\newtheorem{dfn}{Definition}

\newtheorem{rmk}{Remark}

\newcommand{\Cs}{{$C^*$-algebra}}

\DeclareMathOperator{\locdim}{locdim}
\DeclareMathOperator{\ord}{ord}
\DeclareMathOperator{\sr}{sr}
\DeclareMathOperator{\csr}{csr}
\DeclareMathOperator{\rr}{rr}
\DeclareMathOperator{\topdim}{topdim}
\DeclareMathOperator{\Prim}{Prim}

\newcommand{\termdef}[1]{\emph{#1}\index{#1}}

\begin{document}

\title{The topological dimension of type I $C^*$-algebras}
\author{Hannes Thiel}
\address{Department of Mathematical Sciences, University of Copenhagen, Universitetsparken 5, DK-2100, Copenhagen \O, Denmark}
\email{thiel@math.ku.dk}

\thanks{
This research was supported by the Danish National Research Foundation through the Centre for Symmetry and Deformation.
}

\subjclass[2010]%
{Primary
46L05, 
46L85; 
Secondary
54F45, 
55M10. 
}


\date{\today}

\begin{abstract}
    While there is only one natural dimension concept for separable, metric spaces, the theory of dimension in noncommutative topology ramifies into different important concepts.
    To accommodate this, we introduce the abstract notion of a noncommutative dimension theory by proposing a natural set of axioms.
    These axioms are inspired by properties of commutative dimension theory, and they are for instance satisfied by the real and stable rank, the decomposition rank and the nuclear dimension.

    We add another theory to this list by showing that the topological dimension, as introduced by Brown and Pedersen, is a noncommutative dimension theory of type I $C^*$-algebras.
    We also give estimates of the real and stable rank of a type I $C^*$-algebra in terms of its topological dimension.
\end{abstract}

\maketitle

\section{Introduction}

\noindent
    The covering dimension of a topological space is a natural concept that extends our intuitive understanding that a point is zero-dimensional, a line is one-dimensional etc.
    While there also exist other dimension theories for topological spaces (e.g., small and large inductive dimension), they all agree for separable, metric spaces.

    This is in contrast to noncommutative topology where the concept of dimension ramifies into different important theories, such as the real and stable rank, the decomposition rank and the nuclear dimension.
    Each of these concepts has been studied in its own right, and they have applications in many different areas.
    A low dimension in each of these theories can be considered as a regularity property, and such regularity properties play an important role in the classification program of $C^*$-algebras, see \cite{Ror2006}, \cite{EllTom2008}, \cite{Win2012} and the references therein.

    In \autoref{HTsec:NCDimThy} of this paper we introduce the abstract notion of a noncommutative dimension theory as an assignment $d\colon\mathcal{C}\to\overline{\mathbb{N}}$ from a class of $C^*$-algebras to the extended natural numbers $\overline{\mathbb{N}}=\{0,1,2,\ldots,\infty\}$ satisfying a natural set of axioms, see \autoref{HTdfn:NCDimThy:ncDimThy}.
    These axioms are inspired by properties of the theory of covering dimension, see \autoref{HTpgr:NCDimThy:motivation}, and they hold for the theories mentioned above.
    Thus, the proposed axioms do not define a unique dimension theory of $C^*$-algebras, but rather they collect the essential properties that such theories (should) satisfy.

    Besides the very plausible axioms (D1)-(D4), we also propose (D5) which means that the property of being at most $n$-dimensional is preserved under approximation by sub-$C^*$-algebras, see \ref{HTpgr:prelim:approximation}.
    This is the noncommutative analog of the notion of ``likeness'', see \ref{HTpgr:prelim:P-like} and \cite[3.1 - 3.3]{Thi2011}.
    This axiom implies that dimension does not increase when passing to the limit of an inductive system of $C^*$-algebras, i.e., $d(\varinjlim A_i)\leq\liminf d(A_i)$, see \autoref{HTprp:NCDimThy:dim_limit}.

    Finally, axiom (D6) says that every \emph{separable} sub-$C^*$-algebra $C\subset A$ is contained in a \emph{separable} sub-$C^*$-algebra $D\subset A$ such that $d(D)\leq d(A)$.
    This is the noncommutative analog of Marde\v{s}i\'{c}'s factorization theorem, which says that every map $f\colon X\to Y$ from a compact space $X$ to a compact, \emph{metrizable} space $Y$ can be factorized through a compact, \emph{metrizable} space $Z$ with $\dim(Z)\leq\dim(X)$, see \autoref{HTpgr:NCDimThy:motivation} and \cite[Corollary 27.5, p.159]{Nag1970} or \cite[Lemma 4]{Mar1960}.

    In \autoref{HTsec:topdim} we show that the topological dimension as introduced by Brown and Pedersen, \cite{BroPed2009}, is a dimension theory in the sense of \autoref{HTdfn:NCDimThy:ncDimThy} for the class of type I $C^*$-algebras.
    The idea of the topological dimension is to simply consider the dimension of the primitive ideal space of a $C^*$-algebra.
    This will, however, run into problems if the primitive ideal space is not Hausdorff.
    One therefore has to restrict to (locally closed) Hausdorff subsets, and taking the supremum over the dimension of these Hausdorff subsets defines the topological dimension, see \autoref{HTdfn:topdim:topdim}.

    In \autoref{HTsec:typeI} we show how to estimate the real and stable rank of a type I $C^*$-algebra in terms of its topological dimension.
    \\

    \autoref{HTsec:typeI} of this article is based on the diploma thesis of the author, \cite{Thi2009}, which was written under the supervision of Wilhelm Winter at the University of M\"unster in 2009.
    Sections \ref{HTsec:NCDimThy} and \ref{HTsec:topdim} are based upon unpublished notes by the author for the masterclass ``The nuclear dimension of $C^*$-algebras'', held at the University of Copenhagen in November 2011.

\section{Preliminaries}
\label{HTsec:Prelim}

\noindent
    We denote by $\mathcal{C}^{*}$ the category of $C^*$-algebras with ${}^*$-homomorphism as morphisms.
    In general, by a morphism between $C^*$-algebras we mean a ${}^*$-homomorphism.

    We write $J\lhd A$ to indicate that $J$ is an ideal in $A$,
    and by an ideal of a $C^*$-algebra we understand a closed, two-sided ideal.
    Given a $C^*$-algebra $A$, we denote by $A_+$ the set of positive elements.
    We denote the minimal unitization of $A$ by $\widetilde{A}$.
    The primitive ideal space of $A$ will be denoted by $\Prim(A)$, and the spectrum by $\widehat{A}$.
    We refer the reader to Blackadar's book, \cite{Bla2006}, for details on the theory of $C^*$-algebras.

    If $F,G\subset A$ are two subsets of a $C^*$-algebra, and $\varepsilon>0$, then we write $F\subset_\varepsilon G$ if for every $x\in F$ there exists some $y\in G$ such that $\|x-y\|<\varepsilon$.
    Given elements $a,b$ in a $C^*$-algebra, we write $a=_\varepsilon b$ if $\|a-b\|<\varepsilon$.
    Given $a,b\in A_+$, we write $a\ll b$ if $b$ acts as a unit for $a$, i.e., $ab=a$, and we write $a\ll_\varepsilon b$ if $ab=_\varepsilon a$.

    We denote by $M_k$ the $C^*$-algebra of $k$-by-$k$ matrices, and by $\mathbb{K}$ the $C^*$-algebra of compact operators on an infinite-dimensional, separable Hilbert space.
    We denote by $\overline{\mathbb{N}}=\{0,1,2,\ldots,\infty\}$ the extended natural numbers.

\begin{pgr}
\label{HTpgr:prelim:pointed_spaces}
    As pointed out in \cite[II.2.2.7, p.61]{Bla2006}, the full subcategory of commutative $C^*$-algebras is dually equivalent to the category $\mathcal{SP}_{*}$ whose objects are pointed, compact Hausdorff spaces and whose morphisms are pointed, continuous maps.

    For a locally compact, Hausdorff space $X$, let $\alpha X$ be its one-point compactification.
    Let $X^+$ be the space with one additional point $x_\infty$ attached, i.e., $X^+=X\sqcup \{x_\infty\}$ if $X$ is compact, and $X^+=\alpha X$ if $X$ is not compact.
    In both cases, the basepoint of $X^+$ is the attached point $x_\infty$.
%
\end{pgr}

\begin{pgr}
\label{HTpgr:prelim:covering_dimension}
    Let $X$ be a space, and let $\mathcal{U}$ be a cover of $X$.
    The \termdef{order} of $\mathcal{U}$, denoted by $\ord(\mathcal{U})$, is the largest integer $k$ such that some point $x\in X$ is contained in $k$ different elements of $\mathcal{U}$ (and $\ord(\mathcal{U})=\infty$ if no such $k$ exists).
    The \termdef{covering dimension} of $X$, denoted by $\dim(X)$, is the smallest integer $n\geq 0$ such that every finite, open cover of $X$ can be refined by a finite, open cover that has order at most $n+1$ (and $\dim(X)=\infty$ if no such $n$ exists).
    We refer the reader to chapter $2$ of Nagami's book \cite{Nag1970} for more details.

    It was pointed out by Morita, \cite{Mor1975}, that in general this definition of covering dimension should be modified to consider only \emph{normal}, finite, open covers.
    However, for normal spaces (e.g. compact spaces) every finite, open cover is normal, so that we may use the original definition.

    The \termdef{local covering dimension} of $X$, denoted by $\locdim(X)$, is the smallest integer $n\geq 0$ such that every point $x\in X$ is contained in a closed neighborhood $F$ such that $\dim(F)\leq n$ (and $\locdim(X)=\infty$ if no such $n$ exists).
    We refer the reader to \cite{Dow1955} and \cite[Chapter 5]{Pea1975} for more information about the local covering dimension.

    It was noted by Brown and Pedersen, \cite[Section 2.2 (ii)]{BroPed2009}, that $\locdim(X)=\dim(\alpha X)$ for a locally compact, Hausdorff space $X$.
    We propose that the natural dimension of a pointed space $(X,x_\infty)\in\mathcal{SP}_{*}$ is $\dim(X)=\locdim(X\setminus\{x_\infty\})$.
    Then, for a commutative $C^*$-algebra $A$, the natural dimension is $\locdim(\Prim(A))$.

    If $G\subset X$ is an open subset of a locally compact space, then $\locdim(G)\leq\locdim(X)$, see \cite[4.1]{Dow1955}.
    It was also shown by Dowker that this does not hold for the usual covering dimension (of non-normal spaces).
\end{pgr}

\begin{pgr}
\label{HTpgr:prelim:approximation}
\index{Approximation by subalgebras}
    A family of sub-$C^*$-algebras $A_i\subset A$ is said to \termdef{approximate} a $C^*$-algebra $A$ (in the literature there also appears the formulation that the $A_i$ ``locally approximate'' $A$), if for every finite subset $F\subset A$, and every $\varepsilon>0$, there exists some $i$ such that $F\subset_\varepsilon A_i$.
    Let us mention some facts about approximation by subalgebras:
    \begin{enumerate}
        \item
        If $A_1\subset A_2\subset \ldots \subset A$
        is an increasing sequence of sub-$C^*$-algebras with $A=\overline{\bigcup_k A_k}$,
        then $A$ is approximated by the family $\{A_k\}$.
        \item
        If $A$ is approximated by a family $\{A_i\}$, and $J\lhd A$ is an ideal, then $J$ is approximated by the family $\{A_i\cap J\}$.
        In particular, if $A=\overline{\bigcup_k A_k}$, then $J=\overline{\bigcup_k (A_k\cap J)}$.

        Similarly, $A/J$ is approximated by the family $\{A_i/(A_i\cap J)\}$.
        \item
        If $A$ is approximated by a family $\{A_i\}$, and $B\subset A$ is a hereditary sub-$C^*$-algebra, then $B$ might \emph{not} be approximated by the family $\{A_i\cap B\}$.
        Nevertheless, $B$ is approximated by algebras that are isomorphic to hereditary sub-$C^*$-algebras of the algebras $A_i$, see \autoref{HTprp:NCDimThy:ApproxHer}.
    \end{enumerate}
\end{pgr}

\begin{pgr}
\label{HTpgr:prelim:P-like}
\index{P-like}
    Let $\mathcal{P}$ be some property of $C^*$-algebras.
    We say that a $C^*$-algebra $A$ is \termdef{$\mathcal{P}$-like} (in the literature there also appears the formulation $A$ is ``locally $\mathcal{P}$'') if $A$ is approximated by subalgebras with property $\mathcal{P}$, see \cite[3.1 - 3.3]{Thi2011}.
    This is motivated by the concept of $\mathcal{P}$-likeness for commutative spaces, as defined in \cite[Definition 1]{MarSeg1963} and further developed in \cite{MarMat1992}.

    We will work in the category $\mathcal{SP}_{*}$ of pointed, compact spaces, see \ref{HTpgr:prelim:pointed_spaces}.
    Let $\mathcal{P}$ be a non-empty class of spaces.
    Then, a space $X\in\mathcal{SP}_{*}$ is said to be $\mathcal{P}$-like if for every finite, open cover $\mathcal{U}$ of $X$ there exists a (pointed) map $f\colon X\to Y$ onto some space $Y\in\mathcal{P}$ and a finite, open cover $\mathcal{V}$ of $Y$ such that $\mathcal{U}$ is refined by $f^{-1}(\mathcal{V})=\{f^{-1}(V)\: |\: V\in\mathcal{V}\}$.

    Note that we have used $\mathcal{P}$ to denote both a class of spaces and a property that spaces might enjoy.
    These are just different viewpoints, as we can naturally assign to a property the class of spaces with that property, and vice versa to each class of spaces the property of lying in that class.

    For commutative $C^*$-algebras, the notion of $\mathcal{P}$-likeness for $C^*$-algebras coincides with that for spaces.
    More precisely, it is shown in \cite[Proposition 3.4]{Thi2011} that for a space $(X,x_\infty)\in\mathcal{SP}_{*}$ and a collection $\mathcal{P}\subset\mathcal{SP}_{*}$, the following are equivalent:
    \begin{enumerate}[(a)  ]
        \item
        $(X,x_\infty)$ is $\mathcal{P}$-like,
        \item
        $C_0(X\setminus\{x_\infty\})$ is approximated by sub-$C^*$-algebras $C_0(Y\setminus\{y_\infty\})$ with $(Y,y_\infty)\in\mathcal{P}$.
    \end{enumerate}

    We note that the definition of covering dimension can be rephrased as follows.
    Let $\mathcal{P}_k$ be the collection of all $k$-dimensional polyhedra (polyhedra are defined by combinatoric data, and their dimension is defined by this combinatoric data).
    Then a compact space $X$ satisfies $\dim(X)\leq k$ if and only if it is $\mathcal{P}_k$-like.
    This motivates (D5) in \autoref{HTdfn:NCDimThy:ncDimThy} below.
\end{pgr}

\begin{pgr}
\label{HTpgr:prelim:CT}
    For the definition of continuous trace $C^*$-algebras we refer to \cite[Definition IV.1.4.12, p.333]{Bla2006}.
    It is known that a $C^*$-algebra $A$ has continuous trace if and only if its spectrum $\widehat{A}$ is Hausdorff and it satisfies Fell's condition, i.e., for every $\pi\in\widehat{A}$ there exists a neighborhood $U\subset\widehat{A}$ of $\pi$ and some $a\in A_+$ such that $\rho(a)$ is a rank-one projection for each $\rho\in U$, see
    \cite[Proposition IV.1.4.18, p.335]{Bla2006}.
\end{pgr}

\begin{pgr}
\label{HTpgr:prelim:typeI}
    A $C^*$-algebra $A$ is called a \termdef{CCR algebra} (sometimes called a liminal algebra) if for each of its irreducible representations $\pi\colon A\to B(H)$ we have that $\pi$ takes values inside the compact operators $K(H)$.

    A \termdef{composition series} for a $C^*$-algebra $A$ is a collection of ideals $J_\alpha\lhd A$, indexed over all ordinal numbers $\alpha\leq\mu$ for some $\mu$, such that $A=J_\mu$ and:
    \begin{enumerate}[(i)   ]
        \item
        if $\alpha\leq\beta$, then $J_\alpha\subset J_\beta$,
        \item
        if $\alpha$ is a limit ordinal, then $J_\alpha=\overline{\bigcup_{\gamma<\alpha}J_\gamma}$.
    \end{enumerate}
    The $C^*$-algebras $J_{\alpha+1}/J_\alpha$ are called the successive quotients of the composition series.

    A $C^*$-algebra is called a \termdef{type I algebra} (sometimes also called postliminal) if it has a composition series with successive quotients that are CCR algebras.
    As it turns out, this is equivalent to having a composition series whose successive quotients have continuous trace.

    For information about type I $C^*$-algebras and their rich structure we refer the reader to Chapter IV.1 of Blackadar's book, \cite{Bla2006}, and Chapter 6 of Pedersen's book, \cite{Ped1979}.
\end{pgr}

\section{Dimension theories for $C^*$-algebras}
\label{HTsec:NCDimThy}

\noindent
    In this section, we introduce the notion of a noncommutative dimension theory by proposing a natural set of axioms that such theories should satisfy.
    These axioms hold for many well-known theories, in particular the real and stable rank, the decomposition rank and the nuclear dimension, see \autoref{HTprp:NCDimThy:axiomatic_dimThy},
    and this will also be discussed more thoroughly in a forthcoming paper.
    In \autoref{HTsec:topdim} we will show that the topological dimension is a dimension theory for type I $C^*$-algebras.

    Our axioms of a noncommutative dimension theory are inspired by facts that the theory of covering dimension satisfies, see \autoref{HTpgr:NCDimThy:motivation}.

    In \autoref{HTdfn:NCDimThy:Morita-inv} we introduce the notion of Morita-invariance for dimension theories.
    If a dimension theory is only defined on a subclass of $C^*$-algebras, then there is a natural extension of the theory to all $C^*$-algebras, see \autoref{HTprp:NCDimThy:DimThy_extension_by_approx}.
    We will show that this extension preserves Morita-invariance.
    \\

    We denote by $\mathcal{C}^{*}$ the category of $C^*$-algebras, and we will use $\mathcal{C}$ to denote a class of $C^*$-algebras.
    We may think of $\mathcal{C}$ as a full subcategory of $\mathcal{C}^{*}$.

\begin{dfn}
\label{HTdfn:NCDimThy:ncDimThy}
\index{Dimension theory}
    Let $\mathcal{C}$ be a class of $C^*$-algebras that is closed under ${}^*$-isomorphisms, and closed under taking ideals, quotients, finite direct sums, and minimal unitizations.
    A \termdef{dimension theory} for $\mathcal{C}$ is an assignment $d\colon\mathcal{C}\to\overline{\mathbb{N}}=\{0,1,2,\ldots,\infty\}$ such that $d(A)=d(A')$ whenever $A,A'$ are isomorphic $C^*$-algebras in $\mathcal{C}$, and moreover the following axioms are satisfied:
    \begin{enumerate}[(D1)   ]
        \item
        $d(J)\leq d(A)$ whenever $J\lhd A$ is an ideal in $A\in\mathcal{C}$,
        \item
        $d(A/J)\leq d(A)$ whenever $J\lhd A\in\mathcal{C}$,
        \item
        $d(A\oplus B)=\max\{d(A),d(B)\}$, whenever $A,B\in\mathcal{C}$,
        \item
        $d(\widetilde{A})=d(A)$, whenever $A\in\mathcal{C}$.
        \item
        If $A\in\mathcal{C}$ is approximated by subalgebras $A_i\in\mathcal{C}$ with $d(A_i)\leq n$, then $d(A)\leq n$.
        \item
        Given $A\in\mathcal{C}$ and a separable sub-$C^*$-algebra $C\subset A$, there exists a separable $C^*$-algebra $D\in\mathcal{C}$ such that $C\subset D\subset A$ and $d(D)\leq d(A)$.
    \end{enumerate}
    Note that we do not assume that $\mathcal{C}$ is closed under approximation by sub-$C^*$-algebra, so that the assumption $A\in\mathcal{C}$ in (D5) is necessary.
    Moreover, in axiom (D6), we do not assume that the separable subalgebra $C$ lies in $\mathcal{C}$.
\end{dfn}

\begin{rmk}
\label{HTpgr:NCDimThy:motivation}
    The axioms in \autoref{HTdfn:NCDimThy:ncDimThy} are inspired by well-known facts of the local covering dimension of commutative spaces, see \ref{HTpgr:prelim:covering_dimension}.

    Axiom (D1) and (D2) generalize the fact that the local covering dimension does not increase when passing to an open (resp. closed) subspace, see \cite[4.1, 3.1]{Dow1955}, and axiom (D3) generalizes the fact that $\locdim(X\sqcup Y)=\max\{\locdim(X),\locdim(Y)\}$.
    Axiom (D4) generalizes that $\locdim(X)=\locdim(\alpha X)$, where $\alpha X$ is the one-point compactification of $X$.

    Axiom (D5) generalizes the fact that a (compact) space is $n$-dimensional if it is $\mathcal{P}_n$-like for the class $\mathcal{P}_n$ of $n$-dimensional spaces, see \ref{HTpgr:prelim:P-like}.
    Note also that \autoref{HTprp:NCDimThy:dim_limit} generalizes the fact that $\dim(\varprojlim X_i)\leq\liminf_i \dim(X_i)$ for an inverse system of compact spaces $X_i$.

    Axiom (D6) is a generalization of the following factorization theorem, due to Mard\-e\v{s}i\'c, see \cite[Corollary 27.5, p.159]{Nag1970} or \cite[Lemma 4]{Mar1960}:
    Given a compact space $X$ and a map $f\colon X\to Y$ to a compact, metrizable space $Y$, there exists a compact, metrizable space $Z$ and maps $g\colon X\to Z, h\colon Z\to Y$ such that $g$ is onto, $\dim(Z)\leq\dim(X)$ and $f=h\circ g$.
    This generalizes (D6), since a unital, commutative $C^*$-algebra $C(X)$ is separable if and only if $X$ is metrizable.

    Axioms (D5) and (D6) are also related to the following concept which is due to Blackadar, \cite[Definition II.8.5.1, p.176]{Bla2006}:
    A property $\mathcal{P}$ of $C^*$-algebras is called \termdef{separably inheritable} if:
    \begin{enumerate}
        \item
        For every $C^*$-algebra $A$ with property $\mathcal{P}$ and separable sub-$C^*$-algebra $C\subset A$, there exists a separable sub-$C^*$-algebra $D\subset A$ that contains $C$ and has property $\mathcal{P}$.
        \item
        Given an inductive system $(A_k,\varphi_k)$ of separable $C^*$-algebras with injective connecting morphisms $\varphi_k\colon A_k\to A_{k+1}$,
        if each $A_k$ has property $\mathcal{P}$, then does the inductive limit $\varinjlim A_k$.
    \end{enumerate}

    Thus, for a dimension theory $d$, the property ``$d(A)\leq n$'' is separably inheritable.

    Axioms (D5) and (D6) imply that $d(A)\leq n$ if and only if $A$ can be written as an inductive limit (with injective connecting morphisms) of separable $C^*$-algebras $B$ with $d(B)\leq n$.
    This allows us to reduce essentially every question about dimension theories to the case of separable $C^*$-algebras.

    By explaining the analogs of (D1)-(D6) for pointed, compact spaces, we have shown the following:
\end{rmk}

\begin{prp}
\label{HTprp:NCDimThy:commutative}
    Let $\mathcal{C}^{*}_{\mathrm{ab}}$ denote the class of commutative $C^*$-algebras.
    Then, the assignment $d\colon\mathcal{C}^{*}_{\mathrm{ab}}\to\overline{\mathbb{N}}$, $d(A):=\locdim(\Prim(A))$, is a dimension theory.
\end{prp}

\begin{rmk}
\label{HTprp:NCDimThy:axiomatic_dimThy}
    We do not suggest that the axioms of \autoref{HTdfn:NCDimThy:ncDimThy} uniquely define a dimension theory.
    This is clear since the axioms do not even rule out the assignments that give each $C^*$-algebra the same value.

    More interestingly,
    the following well-known theories are dimension theories for the class of all $C^*$-algebras:
    \begin{enumerate}
        \item
        The stable rank as defined by Rieffel, \cite[Definition 1.4]{Rie1983}.
        \item
        The real rank as introduced by Brown and Pedersen, \cite{BroPed1991}.
        \item
        The decomposition rank of Kirchberg and Winter, \cite[Definition 3.1]{KirWin2004}.
        \item
        The nuclear dimension of  Winter and Zacharias, \cite[Definition 2.1]{WinZac2010}.
    \end{enumerate}
    Indeed, for the real and stable rank, (D1) and (D2) are proven in \cite[Th\'{e}or\`{e}me 1.4]{Elh1995} and \cite[Theorems 4.3, 4.4]{Rie1983}.
    Axiom (D3) is easily verified, and (D4) holds by definition.
    It is shown in \cite[Theorem 5.1]{Rie1983} that (D5) holds in the special case of an approximation by a countable inductive limit, but the same argument works for general approximations and also for the real rank.
    Finally, it is noted in \cite[II.8.5.5, p.178]{Bla2006} that (D6) holds.

    For the nuclear dimension, axioms (D1), (D2), (D3), (D6) and (D4) follow from Propositions 2.5, 2.3, 2.6 and Remark 2.11 in \cite{WinZac2010}, and (D5) is easily verified.
    For the decomposition rank, (D5) is also easily verified, and axiom (D6) follows from \cite[Proposition 2.6]{WinZac2010} adapted for c.p.c. approximations instead of c.p. approximations.
    The other axioms (D1)-(D4) follow from Proposition 3.8, 3.11 and Remark 3.2 of \cite{KirWin2004} for separable \Cs{s}.
    Using axioms (D5) and (D6) this can be extended to all \Cs{s}.

    Thus, the idea of \autoref{HTdfn:NCDimThy:ncDimThy} is to collect the essential properties that many different noncommutative dimension theories satisfy.
    Our way of axiomatizing noncommutative dimension theories should therefore not be confused with the work on axiomatizing the dimension theory of metrizable spaces, see e.g. \cite{Nis1974} or \cite{Cha1994}, since these works pursue the goal of finding axioms that uniquely characterize covering dimension.
\end{rmk}

\begin{prp}
\label{HTprp:NCDimThy:dim_limit}
    Let $d\colon\mathcal{C}\to\overline{\mathbb{N}}$ be a dimension theory,
    and let $(A_i,\varphi_{i,j})$ be an inductive system with $A_i\in\mathcal{C}$ and such that the limit $A:=\varinjlim A_i$ also lies in $\mathcal{C}$.
    Then $d(A)\leq\liminf_i d(A_i)$.
\end{prp}
\begin{proof}
    See \cite[II.8.2.1, p.156]{Bla2006} for details about inductive systems and inductive limits.
    For each $i$, let $\varphi_{\infty,i}\colon A_i\to A$ denote the natural morphism into the inductive limit.
    Then the subalgebra $\varphi_{\infty,i}(A_i)\subset A$ is a quotient of $A_i$, and therefore $d(\varphi_{\infty,i}(A_i))\leq d(A_i)$ by (D2).
    If $J\subset I$ is cofinal, then $A$ is approximated by the collection of subalgebras $(\varphi_{\infty,i}(A_i))_{i\in J}$.
    It follows from (D5) that $d(A)$ is bounded by $\sup_{i\in J} d(A_i)$.
    Since this holds for each cofinal subset $J\subset I$, we obtain:
    \begin{align*}
        d(A)
            &\leq\inf\{\sup_{i\in J} d(A_i) \: |\: J\subset I \text{ cofinal}\}
            \ =\liminf_i d(A_i),
    \end{align*}
    as desired.
\end{proof}

\begin{lma}
\label{HTprp:NCDimThy:sepSub_fullHer}
    Let $A$ be a $C^*$-algebra, let $B\subset A$ be a full, hereditary sub-$C^*$-algebra, and let $C\subset A$ be a separable sub-$C^*$-algebra.
    Then there exists a separable sub-$C^*$-algebra $D\subset A$ containing $C$ such that $D\cap B\subset D$ is full, hereditary.
\end{lma}
\begin{proof}
    The proof is inspired by the proof of \cite[Proposition 2.2]{Bla1978}, see also \cite[Theorem II.8.5.6, p.178]{Bla2006}.
    We inductively define separable sub-$C^*$-algebras $D_k\subset A$.
    Set $D_1:=C$, and assume $D_k$ has been constructed.
    Let $S_k:=\{x_1^k,x_2^k,\ldots\}$ be a countable, dense subset of $D_k$.
    Since $B$ is full in $A$, there exist for each $i\geq 1$ finitely many elements $a_{i,j}^k,c_{i,j}^k\in A$ and $b_{i,j}^k\in B$ such that
    \begin{align*}
        \|x_i^k - \sum_j a_{i,j}^k b_{i,j}^k c_{i,j}^k \| <1/k.
    \end{align*}

    Set $D_{k+1}:=C^*(D_k,a_{i,j}^k,b_{i,j}^k,c_{i,j}^k, i,j\geq 1)$.
    Then define $D:=\overline{\bigcup_k D_k}$, which is a separable sub-$C^*$-algebra of $A$ containing $C$.

    Note that $D\cap B\subset D$ is a hereditary sub-$C^*$-algebra, and let us check that it is also full.
    We need to show that the linear span of $D(D\cap B)D$ is dense in $D$.
    Let $d\in D$ and $\varepsilon>0$ be given.
    Note that $\bigcup_k S_k$ is dense in $D$.
    Thus, we may find $k$ and $i$ such that $\|d-x_i^k\|<\varepsilon/2$.
    We may assume $k\geq 2/\varepsilon$.
    By construction, there are elements $a_{i,j}^k,c_{i,j}^k\in D_{k+1}$ and $b_{i,j}^k\in B\cap D_{k+1}$ such that
    $\|x_i^k - \sum_j a_{i,j}^k b_{i,j}^k c_{i,j}^k \| <1/k$.
    It follows that the distance from $d$ to the closed linear span of $D(D\cap B)D$ is at most $\varepsilon$.
    Since $d$ and $\varepsilon$ were chosen arbitrarily, this shows that $D\cap B\subset D$ is full.
\end{proof}

\begin{prp}
\label{HTprp:NCDimThy:TFAE_Morita}
    Let $d\colon\mathcal{C}^{*}\to\overline{\mathbb{N}}$ be a dimension theory.
	Then the following statements are equivalent:
	\begin{enumerate}[(1)  ]
		\item
            For all $C^*$-algebras $A, B$: If $B\subset A$ is a full, hereditary sub-$C^*$-algebra, then $d(A)=d(B)$.
		\item
			For all $C^*$-algebras $A, B$: If $A$ and $B$ are Morita equivalent, then $d(A)=d(B)$.
		\item
			For all $C^*$-algebras $A$: $d(A)=d(A\otimes\mathbb{K})$.
	\end{enumerate}
    Moreover, each of the statements is equivalent to the (a priori weaker) statement where the appearing $C^*$-algebras are additionally assumed to be separable.
		
    If $d$ satisfies the above conditions, and $B\subset A$ is a (not necessarily full) hereditary sub-$C^*$-algebra, then $d(B)\leq d(A)$.
\end{prp}
\begin{proof}
    For each of the statements $(1),(2),(3)$, let us denote the statement where the appearing $C^*$-algebras are assumed to be separable by $(1s),(2s),(3s)$ respectively.
    For example:
	\begin{enumerate}
		\item[$(3s)$  ]
			For all separable $C^*$-algebras $A$: $d(A)=d(A\otimes\mathbb{K})$.
	\end{enumerate}

    The implications ``$(1)\Rightarrow(1s)$'', ``$(2)\Rightarrow(2s)$'', and ``$(3)\Rightarrow(3s)$'' are clear.
    The implication ``$(2s)\Rightarrow(3s)$'' follows since $A$ and $A\otimes\mathbb{K}$ are Morita equivalent, and ``$(1s)\Rightarrow(3s)$'' follows since $A\subset A\otimes\mathbb{K}$ is a full, hereditary sub-$C^*$-algebra.

    It remains to show the implication ``$(3s)\Rightarrow(1)$''.
    Let $A$ be a $C^*$-algebra, and let $B\subset A$ be a full, hereditary sub-$C^*$-algebra.
    We need to show $d(A)=d(B)$.
    To that end, we will construct separable sub-$C^*$-algebras $A'\subset A$ and $B'\subset B$
    that approximate $A$ and $B$, respectively,
    and such that $d(A')=d(B')\leq\min\{d(A), d(B)\}$.
    Together with (D5), this implies $d(A)=d(B)$.

    So let $F\subset A$ and $G\subset B$ be finite sets.
    We may assume $G\subset F$.
    We want to find $A'$ and $B'$ with the mentioned properties and such that $F\subset A'$ and $G\subset B'$.

    We inductively define separable sub-$C^*$-algebras $C_k,D_k\subset A$ and $E_k\subset B$ such that:
    \begin{enumerate}[(a)  ]
        \item
        $C_k\subset D_k$ and $D_k\cap B\subset D_k$ is full,
        \item
        $D_k\cap B\subset E_k$ and $d(E_k)\leq d(B)$,
        \item
        $E_k,D_k\subset C_{k+1}$ and $d(C_{k+1})\leq d(A)$.
    \end{enumerate}

    We start with $C_1:=C^*(F)\subset A$.
    If $C_k$ has been constructed, we apply \autoref{HTprp:NCDimThy:sepSub_fullHer} to find $D_k$ satisfying $(a)$.
    If $D_k$ has been constructed, we apply (D6) to $D_k\cap B\subset B$ to find $E_k$ satisfying $(b)$.
    If $E_k$ has been constructed, we apply axiom (D6) to $C^*(D_k,E_k)\subset A$ to find $C_{k+1}$ satisfying $(c)$.

    Then let $A':=\overline{\bigcup_k C_k}=\overline{\bigcup_k D_k}$, and $B':=\overline{\bigcup_k (D_k\cap B)}=\overline{\bigcup_k E_k}$.
    The situation is shown in the following diagram:
    \begin{center}
        \makebox{
        \xymatrix@=20pt{
        C_k \ar@{}|{\subset}[r]
        & D_k \ar@{}|<<<<<{\subset}[r] \ar@{}|{\cup}[d]
        & {C^*(D_k,E_k)} \ar@{}|>>>>>{\subset}[r] \ar@{}|{\cup}[d]
        & C_{k+1} \ar@{}|{\subset}[r]
        & \ldots \ar@{}|{\subset}[r]
        & A' \\
        & D_k\cap B \ar@{}|{\subset}[r]
        & E_k \ar@{}|{\subset}[r]
        & \ldots & \ldots \ar@{}|{\subset}[r]
        & B'
        } }
    \end{center}

    Let us verify that $A'$ and $B'$ have the desired properties.
    First, since $d(C_k)\leq d(A)$ for all $k$, we get $d(A')\leq d(A)$ from (D5).
    Similarly, we get $d(B')\leq d(B)$.
    For each $k$ we have that $D_k\cap B\subset D_k$ is a full, hereditary sub-$C^*$-algebra, and therefore the same holds for $B'\subset A'$.
    Since $A'$ and $B'$ are separable (and hence $\sigma$-unital), we may apply Brown's stabilization theorem, \cite[Theorem 2.8]{Bro1977}, and obtain $A'\otimes\mathbb{K}\cong B'\otimes\mathbb{K}$.
    Together with the assumption $(3s)$, we obtain $d(A')=d(A'\otimes\mathbb{K})=d(B'\otimes\mathbb{K})=d(B')$.
    This finishes the construction of $A'$ and $B'$, and we deduce $d(A)=d(B)$ from (D5).

    Lastly, if $d$ satisfies condition (1), and $B\subset A$ is a (not necessarily full) hereditary sub-$C^*$-algebra, then $B$ is full, hereditary in the ideal $J\lhd A$ generated by $B$.
    By (D1) and condition (1) we have $d(B)=d(J)\leq d(A)$.
\end{proof}

\begin{dfn}
\label{HTdfn:NCDimThy:Morita-inv}
\index{Dimension theory!Morita-invariant}
    A dimension theory $d\colon\mathcal{C}^{*}\to\overline{\mathbb{N}}$ is called \termdef{Morita-invariant} if it satisfies the conditions of \autoref{HTprp:NCDimThy:TFAE_Morita}.
\end{dfn}

\noindent
    Given positive elements $a,b$ in a $C^*$-algebra, recall that we write $a=_\sigma b$ if $\|a-b\|<\sigma$.
    We write $a\ll_\sigma b$ if $ab=_\sigma a$.

\begin{lma}
\label{HTprp:NCDimThy:TwistIntoHer}
    For every $\varepsilon>0$ there exists $\delta>0$ with the following property:
    Given a $C^*$-algebra $A$, and contractive elements $a,b\in A_+$ with $a=_\delta b$, there exists a partial isometry $v\in A^{**}$ such that:
    \begin{enumerate}
        \item
        $v(a-\delta)_+v^*\in \overline{bAb}$.
        \item
        If $d\in A_+$ is contractive with $d\ll_\sigma a$, then $vdv^*=_{4\sigma+\varepsilon} d$.
    \end{enumerate}
\end{lma}
\begin{proof}
    To simplify the proof, we will fix $\delta>0$ and verify the statement for $\varepsilon=\varepsilon(\delta)$ with the property that $\varepsilon(\delta)\to 0$ when $\delta\to 0$.

    Fix $\delta>0$.
    Let $A$ be a $C^*$-algebra, and let $a,b\in A_+$ be contractive elements such that $a=_\delta b$.
    Without loss of generality we may assume that $A$ is unital.
    It is well-known that there exists $s\in A$ such that $s(a-\delta)_+s^*\in \overline{bAb}$, see \cite[Proposition 2.4]{Ror1992}.
    One could follow the proof to obtain an estimate similar to that in statement (2).
    It is, however, easier to find $v\in A^{**}$ such that (1) and (2) hold, and for our application in \autoref{HTprp:NCDimThy:ApproxHer} it is sufficient that $v$ lies in $A^{**}$.

    It follows from $a=_\delta b$ that $a-\delta\leq b$, and hence:
    \begin{align*}
        (a-\delta)_+^2
        = (a-\delta)_+^{1/2}(a-\delta)(a-\delta)_+^{1/2}
        \leq (a-\delta)_+^{1/2}b(a-\delta)_+^{1/2}.
    \end{align*}

    Set $z := b^{1/2}(a-\delta)_+^{1/2}$.
    Then:
    \[
    |z| = ((a-\delta)_+^{1/2}b(a-\delta)_+^{1/2})^{1/2},
    \quad\quad
    |z^*| = (b^{1/2}(a-\delta)_+b^{1/2})^{1/2},
    \]
    and we let $z=v|z|$ be the polar decomposition of $z$, with $v\in A^{**}$.
    We claim that $v$ has the desired properties.
    First, note that $v ((a-\delta)_+^{1/2}b(a-\delta)_+^{1/2}) v^* = b^{1/2}(a-\delta)_+b^{1/2} \in \overline{bAb}$, and therefore also $v(a-\delta)_+ v^* \in \overline{bAb}$, which verifies property (1).

    For property (2), let us start by estimating the distance from $a$ to $z$ and $|z|$.
    It is known that there exists an assignment $\sigma\mapsto\varepsilon_1(\sigma)$ with the following property:
    Whenever $x,y$ are positive, contractive elements of a $C^*$-algebra, and $x=_\sigma y$, then $x^{1/2}=_{\varepsilon_1(\sigma)}y^{1/2}$, and moreover $\varepsilon_1(\sigma)\to 0$ as $\sigma\to 0$.
    We may assume $\sigma\leq\varepsilon_1(\sigma)$, and we will use this to simplify some estimates below.

    Then, using $(a-\delta)_+=_\delta a$ and so $(a-\delta)_+^{1/2}=_{\varepsilon_1(\sigma)} a^{1/2}$ at the second step,
    \begin{align}
    \label{eq:HTprp:NCDimThy:ApproxHer:1}
        z
        = b^{1/2}(a-\delta)_+^{1/2}
        =_{\varepsilon_1(\delta)} b^{1/2}a^{1/2}
        =_{\varepsilon_1(\delta)} a.
    \end{align}

    For $|z|$ we compute, using $(a-\delta)_+^{1/2}b(a-\delta)_+^{1/2} =_{3\varepsilon_1(\delta)}a^2$ at the second step,
    \begin{align}
    \label{eq:HTprp:NCDimThy:ApproxHer:2}
        |z|
        = ((a-\delta)_+^{1/2}b(a-\delta)_+^{1/2})^{1/2}
        =_{\varepsilon_1(3\varepsilon_1(\delta))} (a^2)^{1/2}
        = a.
    \end{align}

    Let $d\in A_+$ be contractive with $d\ll_\sigma a$.
    Then $ada=_{2\sigma} d$, and we may estimate the distance from $vdv^*$ to $d$ as follows:
    \begin{align*}
        vdv^*
        =_{2\sigma} vadav^*
        \stackrel{\eqref{eq:HTprp:NCDimThy:ApproxHer:2}}{ =\mathrel{\mkern-3mu}= }_{ 2\varepsilon_1(3\varepsilon_1(\delta)) } v|z|d|z|v^*
        = zdz
        \stackrel{\eqref{eq:HTprp:NCDimThy:ApproxHer:1}}{ =\mathrel{\mkern-3mu}= }_{4\varepsilon_1(\delta)} ada
        =_{2\sigma} d.
    \end{align*}

    Thus, $\|vdv^*-d\|\leq 4\sigma +2\varepsilon_1(3\varepsilon_1(\delta)) +4\varepsilon_1(\delta)$,
    and this distance converges to $4\sigma$ when $\delta\to 0$.
    This completes the proof.
\end{proof}

\begin{prp}
\label{HTprp:NCDimThy:ApproxHer}
    Let $A$ be a $C^*$-algebra, and let $B\subset A$ be a hereditary sub-$C^*$-algebra.
    Assume $A$ is approximated by sub-$C^*$-algebras $A_i\subset A$.
    Then $B$ is approximated by subalgebras that are isomorphic to hereditary sub-$C^*$-algebras of the algebras $A_i$, i.e., given a finite set $F\subset B$ and $\varepsilon>0$, there exists a sub-$C^*$-algebra $B'\subset B$ such that $F\subset_\varepsilon B'$ and $B'$ is isomorphic to a hereditary sub-$C^*$-algebra of $A_i$ for some $i$.
\end{prp}
\begin{proof}
    Let $F\subset B$ and $\varepsilon>0$ be given.
    We let $\gamma=\varepsilon/36$, which is justified by the estimates that we obtain through the course of the proof.
    Without loss of generality, we may assume that $F$ consists of positive, contractive elements.

    There exists $b\in B_+$ such that $b$ almost acts as a unit on the elements of $F$ in the sense that $x\ll_\gamma b$ for all $x\in F$.
    Let $\delta>0$ be the tolerance we get from \autoref{HTprp:NCDimThy:TwistIntoHer} for $\gamma$.
    We may assume $\delta\leq\gamma$, and to simplify the computations below we will often estimate a distance by $\gamma$, even if it could be estimated by $\delta$.

    By assumption, the algebras $A_i$ approximate $A$.
    Thus, there exists $i$ such that there is a positive, contractive element $a\in A_i$ with $a=_\delta b$, and such that for each $x\in F$ there exists a positive, contractive $x'\in A_i$ with $x'=_\delta x$.
    Then:
    \begin{align*}
        x'(a-\delta)_+=_{3\delta}xb=_\gamma x=_\delta x',
    \end{align*}
    and so $x'\ll_{5\gamma}(a-\delta)_+$, since $\delta\leq\gamma$.
    In general, if two positive, contractive elements $s,t$ satisfy $s\ll_\sigma t$, then $s=_{2\sigma}tst\ll_\sigma t$.
    Thus, if for each $x\in F$ we set $x'':=(a-\delta)_+ x' (a-\delta)_+$, then we obtain:
    \begin{align}
    \label{eq:HTprp:NCDimThy:ApproxHer}
        x=_\gamma x'=_{10\gamma} x'' \ll_{5\gamma} (a-\delta)_+.
    \end{align}

    Since $a=_\delta b$, we obtain from \autoref{HTprp:NCDimThy:TwistIntoHer} a partial isometry $v\in A^{**}$ such that $v(a-\delta)_+v^*\in \overline{bAb}$.
    Let $A':=(a-\delta)_+ A_i(a-\delta)_+$, which is a hereditary sub-$C^*$-algebra of $A_i$.
    The map $x\mapsto vxv^*$ defines an isomorphism from $A'$ onto $B':=vA'v^*$.
    Since $B$ is hereditary, $B'$ is a sub-$C^*$-algebra of $B$.
    Let us estimate the distance from $F$ to $B'$.

    For each $x\in F$, we have computed in \eqref{eq:HTprp:NCDimThy:ApproxHer} that $x'' \ll_{5\gamma} (a-\delta)_+$, which implies $x''\ll_{6\gamma}a$.
    From statement (2) of \autoref{HTprp:NCDimThy:TwistIntoHer} we deduce $vx''v^* =_{25\gamma}x''$.
    Altogether, the distance between $x$ and $vx''v^*$ is at most $36\gamma$.
    Since $vx''v^*\in B'$, and since we chose $\gamma=\varepsilon/36$, we have $F\subset_\varepsilon B'$, as desired.
\end{proof}

\begin{prp}
\label{HTprp:NCDimThy:DimThy_extension_by_approx}
    Let $d\colon\mathcal{C}\to\overline{\mathbb{N}}$ be a dimension theory.
    For any $C^*$-algebra $A$ define:
    \begin{align}
    	\widetilde{d}(A) &:= \inf\{k\in\mathbb{N}\: |\: A
        \text{ is approximated by sub-$C^*$-algebras } B\in\mathcal{C} \text{ with } d(B)\leq k\},
    \end{align}
    where we define the infimum of the empty set to be $\infty\in\overline{\mathbb{N}}$.

    Then $\widetilde{d}\colon\mathcal{C}^{*}\to\overline{\mathbb{N}}$ is a dimension theory that agrees with $d$ on $\mathcal{C}$.

    If, moreover, $\mathcal{C}$ is closed under stable isomorphism, and $d(A)=d(A\otimes\mathbb{K})$ for every (separable) $A\in\mathcal{C}$, then $\widetilde{d}$ is Morita-invariant.
\end{prp}
\begin{proof}
    If $A \in\mathcal{C}$, then clearly $\widetilde{d}(A)\leq d(A)$, and the converse inequality follows from axiom (D5).
    Axioms (D1)-(D5) for $\widetilde{d}$ are easy to check.

    Let us check axiom (D6) for $\widetilde{d}$.
    Assume $A$ is a $C^*$-algebra, and assume $C\subset A$ is a separable sub-$C^*$-algebra.
    Set $n:=\widetilde{d}(A)$, which we may assume is finite.
    We need to find a separable sub-$C^*$-algebra $D\subset A$ such that $C\subset D$ and $\widetilde{d}(D)\leq n$.

    We first note the following:
    For a finite set $F\subset A$, and $\varepsilon>0$ we can find a separable sub-$C^*$-algebra $A(F,\varepsilon)\subset A$ with $d(A(F,\varepsilon))\leq n$ and $F\subset_\varepsilon A(F,\varepsilon)$.
    Indeed, by definition of $\widetilde{d}$ we can first find a sub-$C^*$-algebra $B\subset A$ with $d(B)\leq n$ and a finite subset $G\subset B$ such that $F\subset_\varepsilon G$.
    Applying (D6) to $C^*(G)\subset B$, we may find a separable sub-$C^*$-algebra $A(F,\varepsilon)\subset B$ with $d(A(F,\varepsilon))\leq n$ and $C^*(G)\subset A(F,\varepsilon)$, which implies $F\subset_\varepsilon A(F,\varepsilon)$.

    We will inductively define separable sub-$C^*$-algebras $D_k\subset A$ and countable dense subsets $S_k=\{x_1^k,x_2^k,\ldots\}\subset D_k$ as follows:
    We start with $D_1:=C$ and choose any countable dense subset $S_1\subset D_1$.
    If $D_l$ and $S_l$ have been constructed for $l\leq k$, then set:
    \begin{align*}
        D_{k+1}:=C^*(D_k, A(\{x_i^j \: |\: i,j\leq k\},1/k)) \subset A,
    \end{align*}
    and choose any countable dense subset $S_{k+1}=\{x_1^{k+1},x_2^{k+1},\ldots\}\subset D_{k+1}$.

    Set $D:=\overline{\bigcup_k D_k}\subset A$, which is a separable $C^*$-algebra containing $C$.
    Let us check that $\widetilde{d}(D)\leq n$, which means that we have to show that $D$ is approximated by sub-$C^*$-algebras $B\in\mathcal{C}$ with $d(B)\leq n$.

    Note that $\{x_i^j\}_{i,j\geq 1}$ is dense in $D$.
    Thus, if a finite subset $F\subset D$, and $\varepsilon>0$ is given, we may find $k$ such that $F\subset_{\varepsilon/2}\{x_i^j \: |\: i,j\leq k\}$, and we may assume $k>2/\varepsilon$.
    By construction, $D$ contains the sub-$C^*$-algebra $B:=A(\{x_i^j \: |\: i,j\leq k\},1/k)$, which satisfies $d(B)\leq n$ and $\{x_i^j \: |\: i,j\leq k\}\subset_{1/k} B$.
    Then $F\subset_\varepsilon B$, which completes the proof that $\widetilde{d}(D)\leq n$.

    Lastly, assume $\mathcal{C}$ is closed under stable isomorphism, and assume $d(A)=d(A\otimes\mathbb{K})$ for every separable $A\in\mathcal{C}$.
    This implies the following:
    If $A$ is a separable $C^*$-algebra in $\mathcal{C}$, and $B\subset A$ is a hereditary sub-$C^*$-algebra, then $B$ lies in $\mathcal{C}$ and $d(B)\leq d(A)$.

    We want to check condition $(3)$ of \autoref{HTprp:NCDimThy:TFAE_Morita} for $\widetilde{d}$.
    Thus, let a separable $C^*$-algebra $A$ be given.
    We need to check $\widetilde{d}(A)=\widetilde{d}(A\otimes\mathbb{K})$.

    If $\widetilde{d}(A)=\infty$, then clearly $\widetilde{d}(A\otimes\mathbb{K})\leq\widetilde{d}(A)$.
    So assume $n:=\widetilde{d}(A)<\infty$, which means that $A$ is approximated by algebras $A_i\subset A$ with $d(A_i)\leq n$.
    Then $A\otimes\mathbb{K}$ is approximated by the subalgebras $A_i\otimes\mathbb{K}\subset A\otimes\mathbb{K}$, and $d(A_i\otimes\mathbb{K})=d(A_i)\leq n$ by assumption.
    Then $\widetilde{d}(A\otimes\mathbb{K})\leq n=\widetilde{d}(A)$.

    Conversely, if $\widetilde{d}(A\otimes\mathbb{K})=\infty$, then $\widetilde{d}(A)\leq\widetilde{d}(A\otimes\mathbb{K})$.
    So assume $n:=\widetilde{d}(A\otimes\mathbb{K})<\infty$, which means that $A\otimes\mathbb{K}$ is approximated by algebras $A_i\subset A$ with $d(A_i)\leq n$.
    Consider the hereditary sub-$C^*$-algebra $A\otimes e_{1,1}\subset A\otimes\mathbb{K}$, which is isomorphic to $A$.
    By \autoref{HTprp:NCDimThy:ApproxHer}, $A\otimes e_{1,1}$ is approximated by subalgebras $B_j\subset A\otimes e_{1,1}$ such that each $B_j$ is isomorphic to a hereditary sub-$C^*$-algebras of $A_i$, for some $i=i(j)$.
    It follows $d(B_j)\leq n$, and then $\widetilde{d}(A) =\widetilde{d}(A\otimes e_{1,1})\leq n=\widetilde{d}(A\otimes\mathbb{K})$.
    Together we get $\widetilde{d}(A)=\widetilde{d}(A\otimes\mathbb{K})$, as desired.
\end{proof}

\section{Topological dimension} 
\label{HTsec:topdim}

\noindent
    One could try to define a dimension theory by simply considering the dimension of the primitive ideal space of a $C^*$-algebra.
    This will, however, run into problems if the primitive ideal space is not Hausdorff.
    Brown and Pedersen, \cite{BroPed2009}, suggested a way of dealing with this problem by restricting to (locally closed) Hausdorff subsets of $\Prim(A)$, and taking the supremum over the dimension of these Hausdorff subsets.
    This defines the topological dimension of a $C^*$-algebra, see \autoref{HTdfn:topdim:topdim}.

    In this section we will show that the topological dimension is a dimension theory in the sense of \autoref{HTdfn:NCDimThy:ncDimThy} for the class of type I $C^*$-algebras.
    It follows from the work of Brown and Pedersen that axioms (D1)-(D4) are satisfied, and we verify axiom (D5) in \autoref{HTprp:topdim:approx_GCR}.
    We use transfinite induction over the length of a composition series of the type I $C^*$-algebra to verify axiom (D6), see \autoref{HTprp:topdim:sep_subalg_GCR-ideal}.

    See \ref{HTpgr:prelim:typeI} for a short reminder on type I $C^*$-algebras.
    For more details, we refer the reader to Chapter IV.1 of Blackadar's book, \cite{Bla2006}, and Chapter 6 of Pedersen's book, \cite{Ped1979}.

\begin{dfn}[{Brown, Pedersen, \cite[2.2 (iv)]{BroPed2007}}]
     Let $X$ be a topological space.
     We define:
     \begin{enumerate}[(1)   ]
        \item
        A subset $C\subset X$ is called \termdef{locally closed} if there is a closed set $F\subset X$ and an open set $G\subset X$ such that $C=F\cap G$.
        \item
        $X$ is called \termdef{almost Hausdorff} if every non-empty closed subset $F$ contains a non-empty relatively open subset $F\cap G$ (so $F\cap G$ is locally closed in $X$) which is Hausdorff.
     \end{enumerate}
\end{dfn}

\begin{pgr}
\label{HTpgr:topdim:typeI_almT2}
    We could consider locally closed subsets as ``well-placed'' subsets.
    Then, being almost Hausdorff means having enough ``well-placed'' Hausdorff subsets.

    For a $C^*$-algebra $A$, the locally closed subsets of $\Prim(A)$ correspond to ideals of quotients of $A$ (equivalently to quotients of ideals of $A$) up to canonical isomorphism, see \cite[2.2(iii)]{BroPed2007}.
    Therefore, the primitive ideal space of every type I $C^*$-algebra is almost Hausdorff, since every non-zero quotient contains a non-zero ideal that has continuous trace, see \cite[Theorem 6.2.11, p. 200]{Ped1979}, and the primitive ideal space of a continuous trace $C^*$-algebra is Hausdorff.
\end{pgr}

\begin{dfn}[{Brown, Pedersen, \cite[2.2(v)]{BroPed2007}}]
\label{HTdfn:topdim:topdim}
\index{Topological dimension}
    Let $A$ be a $C^*$-algebra.
    If $\Prim(A)$ is almost Hausdorff, then the \termdef{topological dimension} of $A$, denoted by $\topdim(A)$, is:
    \begin{align} \setcounter{equation}{0}
        \topdim(A)  &:= \sup\{\locdim(S) \: |\: S\subset\Prim(A) \text{ locally closed, Hausdorff} \}.
    \end{align}
\end{dfn}

\noindent
    We will now show that the topological dimension satisfies the axioms of \autoref{HTdfn:NCDimThy:ncDimThy}.
    The following result immediately implies (D1)-(D4).

\begin{prp}[{Brown, Pedersen, \cite[Proposition 2.6]{BroPed2007}}]
\label{HTprp:topdim:compSeries}
    Let $(J_\alpha)_{\alpha\leq \mu}$ be a composition series for a $C^*$-algebra $A$.
    Then $\Prim(A)$ is almost Hausdorff if and only if $\Prim(J_{\alpha+1}/J_\alpha)$ is almost Hausdorff for each $\alpha<\mu$, and if this is the case, then:
    \begin{align} \setcounter{equation}{0}
        \topdim(A)
            &=\sup_{\alpha<\mu}\topdim(J_{\alpha+1}/J_\alpha).
    \end{align}
\end{prp}

\noindent
    The following result is implicit in the papers of Brown and Pedersen, e.g. \cite[Theorem 5.6]{BroPed2009}.

\begin{prp}
\label{HTprp:topdim:her}
    Let $A$ be a $C^*$-algebra, and let $B\subset A$ be a hereditary sub-$C^*$-algebra.
    If $\Prim(A)$ is locally Hausdorff, then so is $\Prim(B)$, and then $\topdim(B)\leq\topdim(A)$.
    If $B$ is even full hereditary, then $\topdim(B)=\topdim(A)$.
\end{prp}
\begin{proof}
    In general, if $B\subset A$ is a hereditary sub-$C^*$-algebra, then $\Prim(B)$ is homeomorphic to an open subset of $\Prim(A)$.
    In fact, $\Prim(B)$ is canonically homeomorphic to the primitive ideal space of the ideal generated by $B$, and this corresponds to an open subset of $\Prim(A)$.

    Note that being locally Hausdorff is a property that passes to locally closed subsets, and so it passes from $\Prim(A)$ to $\Prim(B)$.
    Further, every locally closed, Hausdorff subset $S\subset\Prim(B)$ is also locally closed (and Hausdorff) in $\Prim(A)$.
    It follows $\topdim(B)\leq\topdim(A)$.

    If $B$ is full, then $\Prim(B)\cong\Prim(A)$ and therefore $\topdim(B)=\topdim(A)$.
\end{proof}

\begin{lma}
\label{HTprp:topdim:lma_approx_CT}
    Let $A$ be a continuous trace $C^*$-algebra, and let $n\in\mathbb{N}$.
    If $A$ is approximated by sub-$C^*$-algebras with topological dimension at most $n$, then $\topdim(A)\leq n$.
\end{lma}
\begin{proof}
    Since $\Prim(A)$ is Hausdorff, we have $\topdim(A)=\locdim(\Prim(A))$.
    Thus, it is enough to show that every $x\in\Prim(A)$ has a neighborhood $U$ with $\dim(U)\leq n$.
    This will allow us to reduce the problem to the situation that $A$ has a global rank-one projection, i.e., that there exists a full, abelian projection $p\in A$, see \cite[IV.1.4.20, p.335]{Bla2006}, which we do as follows:

    Let $x\in\Prim(A)$ be given.
    Since $A$ has continuous trace, there exists an open neighborhood $U\subset\Prim(A)$ of $x$ and an element $a\in A_+$ such that $\rho(a)$ is a rank-one projection for every $\rho\in U$, see \ref{HTpgr:prelim:CT}.
    Then there exists a closed, compact neighborhood $Y\subset\Prim(A)$ of $x$ that is contained in $U$.
    Let $J\lhd A$ be the ideal corresponding to $\Prim(A)\setminus Y$.
    The image of $a$ in the quotient $A/J$ is a full, abelian projection.
    Since $A$ is approximated by subalgebras $B\subset A$ with $\topdim(B)\leq n$, $A/J$ is approximated by the subalgebras $B/(B\cap J)$ with $\topdim(B/(B\cap J))\leq\topdim(B)\leq n$.
    If we can show that this implies $\dim(Y)=\topdim(A/J)\leq n$, then every point of $\Prim(A)$ has a closed neighborhood of dimension $\leq n$, which means $\topdim(A)=\locdim(\Prim(A))\leq n$.

    We assume from now on that $A$ has continuous trace with a full, abelian projection $p\in A$.
    Thus, $pAp\cong C(X)$ where $X:=\Prim(A)$ is a compact, Hausdorff space.
    Assume $A$ is approximated by subalgebras $A_i\subset A$ with $\topdim(A_i)\leq n$.
    It follows from \autoref{HTprp:NCDimThy:ApproxHer} that the hereditary sub-$C^*$-algebra $pAp$ is approximated by subalgebras $B_j$ such that each $B_j$ is isomorphic to a hereditary sub-$C^*$-algebra of $A_i$, for some $i=i(j)$.
    By \autoref{HTprp:topdim:her}, $\topdim(B_j)\leq\topdim(A_{i(j)})\leq n$ for each $j$.

    Thus, $C(X)$ is approximated by commutative subalgebras $C(X_j)$ with $\dim(X_j)=\topdim(C(X_j))\leq n$.
    It follows from \autoref{HTprp:NCDimThy:commutative} that $\dim(X)\leq n$, as desired.
\end{proof}

\begin{prp}
\label{HTprp:topdim:approx_GCR}
    Let $A$ be a type I $C^*$-algebra, and let $n\in\mathbb{N}$.
    If $A$ is approximated by sub-$C^*$-algebras with topological dimension at most $n$, then $\topdim(A)\leq n$.
\end{prp}
\begin{proof}
    Let $(J_\alpha)_{\alpha\leq \mu}$ be a composition series for $A$ such that each successive quotient has continuous trace, and assume $A$ is approximated by subalgebras $A_i\subset A$ with $\topdim(A_i)\leq n$.

    Then $J_{\alpha+1}/J_\alpha$ is approximated by the subalgebras $(A_i\cap J_{\alpha+1})/(A_i\cap J_\alpha)$, see \ref{HTpgr:prelim:approximation}.
    Since $\topdim((A_i\cap J_{\alpha+1})/(A_i\cap J_\alpha))\leq\topdim(A_i)\leq n$, we obtain from the above \autoref{HTprp:topdim:lma_approx_CT} that $\topdim(J_{\alpha+1}/J_\alpha)\leq n$.
    By \autoref{HTprp:topdim:compSeries}, $\topdim(A)=\sup_{\alpha<\mu}\topdim(J_{\alpha+1}/J_\alpha)\leq n$, as desired.
\end{proof}

\begin{rmk}
    It is noted in \cite[Remark 2.5(v)]{BroPed2007} that a weaker version of \autoref{HTprp:topdim:approx_GCR} would follow from \cite{Sud2004}.
    However, the statement is formulated as an axiom there, and it is not clear that the formulated axioms are consistent and give a dimension theory that agrees with the topological dimension.
\end{rmk}

\noindent
    We will now prove that the topological dimension of type I $C^*$-algebras satisfies the Marde\v{s}i\'c factorization axiom (D6).
    We start with two lemmas.

\begin{lma}
\label{HTprp:topdim:lma_sep_subalg_CT}
    Let $A$ be a continuous trace $C^*$-algebra, and let $C\subset A$ be a separable sub-$C^*$-algebra.
    Then there exists a separable, continuous trace sub-$C^*$-algebra $D\subset A$ that contains $C$, and such that the inclusion $C\subset D$ is proper, and $\topdim(D)\leq\topdim(A)$.
\end{lma}
\begin{proof}
    Let us first reduce to the case that $A$ is $\sigma$-unital, and the inclusion $C\subset A$ is proper.
    To this end, consider the hereditary sub-$C^*$-algebra $A':=CAC\subset A$.
    Since $C$ is separable, it contains a strictly positive element which is then also strictly positive in $A'$.
    Moreover, having continuous trace passes to hereditary sub-$C^*$-algebras, see \cite[Proposition 6.2.10, p.199]{Ped1979}.
    Thus, $A'$ is $\sigma$-unital and $C\subset A'$ is proper.
    Moreover, $\topdim(A')\leq\topdim(A)$ by \autoref{HTprp:topdim:her}.

    Thus, by replacing $A$ with $CAC$, we may assume from now on that $A$ is $\sigma$-unital and that the inclusion $C\subset A$ is proper.
    Set $X:=\Prim(A)$.
    By Brown's stabilization theorem, \cite[Theorem 2.8]{Bro1977}, there exists an isomorphism $\Phi\colon A\otimes\mathbb{K}\to C_0(X)\otimes\mathbb{K}$.
    Let $e_{ij}\in\mathbb{K}$ be the canonical matrix units, and consider the following $C^*$-algebra:
    \begin{align*}
        E &:=C^{*}(\bigcup_{i,j} e_{1i}\Phi(C\otimes\mathbb{K})e_{j1}) \subset C_0(X)\otimes e_{11}.
    \end{align*}
    The following diagram shows some of the $C^*$-algebras and maps that we will construct below:
    \begin{center}
        \makebox{
        \xymatrix@=20pt{
        A\otimes e_{11} \ar@{}|{\cup}[d] \ar@{}|{\subset}[r]
        & A\otimes\mathbb{K} \ar@{}|{\cup}[d] \ar[r]^<<<<<{\Phi}_<<<<<{\cong}
        & C_0(X)\otimes\mathbb{K} \ar@{}|{\cup}[d] \\
        D \ar@{}|{\cup}[d] \ar@{}|{\subset}[r]
        & \Phi^{-1}(D') \ar@{}|{\cup}[d] \ar[r]_<<<<{\cong}
        & C_0(Z_0)\otimes\mathbb{K} \ar@{}|{\cup}[d] \ar@{}|>>>>{=}[r]
        & D' \\
        C\otimes e_{11} \ar@{}|{\subset}[r]
        & C\otimes\mathbb{K} \ar[r]_<<<<<<{\cong}
        & \Phi(C\otimes\mathbb{K})
        } }
    \end{center}

    Note that $E$ is separable and commutative.
    Thus, there exists a separable sub-$C^*$-algebra $C_0(Y)\subset C_0(X)$ such that $E=C_0(Y)\otimes e_{11}$.
    We constructed $E$ such that $\Phi(C\otimes\mathbb{K})\subset C_0(Y)\otimes\mathbb{K}$.

    The inclusion $C_0(Y)\subset C_0(X)$ is induced by a pointed, continuous map $f\colon X^+\to Y^+$, see \ref{HTpgr:prelim:pointed_spaces}.
    Recall that a compact, Hausdorff space $M$ is metrizable if and only if $C(M)$ is separable.
    Thus, $Y^+$ is compact, metrizable.

    By Marde\v{s}i\'c's factorization theorem, see \cite[Corollary 27.5, p.159]{Nag1970} or \cite[Lemma 4]{Mar1960}, there exists a compact, metrizable space $Z$ with $\dim(Z)\leq\dim(X)$ and continuous (surjective) maps $g\colon X\to Z$ and $h\colon Z\to Y$ such that $f=h\circ g$.
    Set $Z_0:=Z\setminus\{g(\infty)\}$, and note that $g^{*}$ induces an embedding $C_0(Z_0)\subset C_0(X)$.
    Moreover, $C_0(Z_0)$ is separable, since $Z$ is compact, metrizable.

    Consider $D':=C_0(Z_0)\otimes\mathbb{K}$.
    We have that $D'$ is a separable, continuous trace $C^*$-algebra such that $\Phi(C\otimes\mathbb{K})\subset C_0(Y)\otimes\mathbb{K}\subset D'$, and $\topdim(D')=\dim(Z)\leq\dim(X)=\topdim(A)$.
    We think of $C$ as included in $C\otimes\mathbb{K}$ via $C\cong C\otimes e_{11}$.
    Set
    \begin{align*}
        D :=(1_{\widetilde{A}}\otimes e_{11}) (\Phi^{-1}(D')) (1_{\widetilde{A}}\otimes e_{11}),
    \end{align*}
    which is a hereditary sub-$C^*$-algebra of $\Phi^{-1}(D')\cong D'$.
    Hence, $D$ is a separable, continuous trace $C^*$-algebra with $\topdim(D)\leq\topdim(D')\leq\topdim(A)$.
    By construction, $C\otimes e_{11}\subset D$, and this inclusion is proper since $D\subset A\otimes e_{11}$ and the inclusion $C\otimes e_{11}\subset A\otimes e_{11}$ is proper.
\end{proof}

\begin{lma}
\label{HTprp:topdim:extension_for_ideal_proper}
    Let $A$ be a $C^*$-algebra, let $J\lhd A$ be an ideal, and let $C\subset A$ be a sub-$C^*$-algebra.
    Assume $K\subset J$ is a sub-$C^*$-algebra that contains $C\cap J$ and such that the inclusion $C\cap J\subset K$ is proper.
    Then $K$ is an ideal in the sub-$C^*$-algebra $C^{*}(K,C)\subset A$ generated by $K$ and $C$.
    Moreover, there is a natural isomorphism $C^{*}(K,C)/K \cong C/(C\cap J)$.
\end{lma}
\begin{proof}
    Set $B:=A/J$ and denote the quotient morphism by $\pi\colon A\to B$.
    Set $D:=\pi(C)\subset B$.
    Clearly, $C^{*}(K,C)$ contains both $K$ and $C$, and it is easy to see that the restriction of $\pi$ to $C^{*}(K,C)$ maps onto $D$.
    The situation is shown in the following commutative diagram, where the top and bottom rows are exact:
    \begin{center}
        \makebox{
        \xymatrix@=20pt{
        0 \ar[r] & J \ar[r] & A \ar[r]^{\pi} & B \ar[r] & 0\\
        & K \ar[r] \ar@{}[u]|{\cup} & {C^{*}(K,C)} \ar@{}[u]|{\cup} \ar[r] & D \ar@{}[u]|{\cup} \\
        0 \ar[r] & C\cap J \ar@{}[u]|{\cup} \ar[r] & C \ar@{}[u]|{\cup} \ar[r] & D \ar@{}[u]|{||} \ar[r] & 0\\
        } }
    \end{center}

    Let us show that $K$ is an ideal in $C^{*}(K,C)$.
    Since $C^{*}(K,C)$ is generated by elements of $K$ and $C$, it is enough to show that $xy$ and $yx$ lie in $K$ whenever $x\in K$ and $y\in K$ or $y\in C$.
    For $y\in K$ that is clear, so assume $y\in C$.

    Since $C\cap J\subset K$ is proper, for any $\varepsilon>0$ there exists $c\in C\cap J$ such that $\|cxc-x\|<\varepsilon$.
    Then $\|xy-cxcy\|,\|yx-ycxc\|<\varepsilon\|y\|$.
    Moreover, $cxcy\in K$ and $ycxc\in K$ since $cy,yc \in C\cap J\subset K$.
    For $\varepsilon>0$ was arbitrary, it follows that $xy,yx\in K$.
    This shows that the middle row in the above diagram is also exact.
\end{proof}

\begin{prp}
\label{HTprp:topdim:sep_subalg_GCR-ideal}
    Let $A$ be a $C^*$-algebra, let $J\lhd A$ be an ideal of type I, and let $C\subset A$ be a separable sub-$C^*$-algebra.
    Then there exists a separable sub-$C^*$-algebra $D\subset A$ such that $C\subset D$ and $\topdim(D\cap J)\leq\topdim(J)$.
\end{prp}
\begin{proof}
    Let $(J_\alpha)_{\alpha\leq\mu}$ be a composition series for $J$ with successive quotients that have continuous trace.
    To simplify notation, we will write $B[\alpha,\beta)$ for $(B\cap J_\beta)/(B\cap J_\alpha)$ and $B[\alpha,\infty)$ for $B/(B\cap J_\alpha)$ whenever $B\subset A$ is a subalgebra and $\alpha\leq\beta\leq\mu$ are ordinals.
    In particular, $A[0,\beta)=J_\beta$ and $A[\alpha,\infty)=A/J_\alpha$.
    We prove the statement of the proposition by transfinite induction over $\mu$, which we carry out in three steps.

    Step 1: The statement holds for $\mu=0$.
    This follows since $J$ is assumed to have a composition series with length $0$ and so $J=\{0\}$ and we can simply set $D:=C$.

    Step 2: If the statement holds for a finite ordinal $n$, then it also holds for $n+1$.

    To prove this, assume $J$ has a composition series $(J_\alpha)_{\alpha\leq n +1}$.
    Let $d:=\topdim(J)$.
    Given $C\subset A$ separable, we want to find a separable subalgebra $D\subset A$ with $C\subset D$ and $\topdim(D[0,n+1))\leq d$.
    The following commutative diagram, whose rows are short exact sequences, contains the algebras and maps that we will construct below:
    \begin{center}
        \makebox{
        \xymatrix@=20pt{
        0 \ar[r] & A[0,1) \ar[r] & A \ar[r] & A[1,\infty) \ar[r] & 0 \\
        0 \ar[r] & E[0,1) \ar[r] \ar@{}[u]|{\cup} & E \ar[r] \ar@{}[u]|{\cup} & E' \ar[r] \ar@{}[u]|{\cup} & 0 \\
        0 \ar[r] & C[0,1) \ar@{}[u]|{\cup} \ar[r] & C \ar@{}[u]|{\cup} \ar[r] & C[1,\infty) \ar[r] \ar@{}[u]|{\cup} & 0 \\
        } }
    \end{center}

    Consider $A[1,\infty)$ together with the ideal $A[1,n+1)=J[1,n+1)$.
    Note that $A[1,n+1)$ has the canonical composition series $(A[1,\alpha))_{1\leq\alpha\leq n+1}$ of length $n$.
    By assumption of the induction, the statement holds for $n$, and so there is a separable sub-$C^*$-algebra $E'\subset A[1,\infty)$ such that $C[1,\infty)\subset E'$ and $\topdim(E'\cap A[1,n+1))\leq\topdim(A[1,n+1))\leq d$.
    Find a separable sub-$C^*$-algebra $E\subset A$ such that $C\subset E$ and $E[1,\infty)=E'$.

    We apply \autoref{HTprp:topdim:lma_sep_subalg_CT} to the inclusion $E[0,1)\subset A[0,1)$ to find a separable sub-$C^*$-algebra $K\subset A[0,1)$ containing $E[0,1)$ and such that the inclusion $E[0,1)\subset K$ is proper, and $\topdim(K)\leq\topdim(A[0,1))\leq d$.
    Set $D:=C^{*}(K,E)\subset A$, which is a separable $C^*$-algebra with $C\subset D$.
    By \autoref{HTprp:topdim:extension_for_ideal_proper}, $D$ is an extension of $E$ by $K$, and therefore \autoref{HTprp:topdim:compSeries} gives:
    \begin{align*}
        \topdim(D[0,n+1))
            &=\max\{\topdim(D[0,1)),\topdim(D[1,n+1))\} \\
            &=\max\{\topdim(K),\topdim(E'\cap A[1,n+1))\} \\
            &\leq d.
    \end{align*}

    Step 3: Assume $\lambda$ is a limit ordinal, and $n$ is finite.
    If the statement holds for all $\alpha<\lambda$, then it holds for $\lambda+n$.

    We will prove this by distinguishing the two sub-cases that $\lambda$ has cofinality at most $\omega$, or cofinality bigger than $\omega$.
    We start the construction for both cases together.
    Later we will treat them separately.
    Let $d:=\topdim(J)$.

    We will inductively define ordinals $\alpha_k<\mu$ and sub-$C^*$-algebras $D_k,E_k\subset A$ with the following properties:
    \begin{enumerate}
    	\item
    	$\alpha_1\leq\alpha_2\leq\ldots$,
    	\item
    	$D_k\subset E_k$ and $\topdim(E_k[\lambda,\lambda+n))\leq d$,
    	\item
    	$E_k\subset D_{k+1}$ and $\topdim(D_{k+1}[0,\alpha_{k+1}))\leq d$.
    \end{enumerate}

    In both cases 3a and 3b below, we construct $E_k$ from $D_k$ as follows:
    Given $D_k$, consider $D_k[\lambda,\infty)\subset A[\lambda,\infty)$ and the ideal $A[\lambda,\lambda+n)\lhd A[\lambda,\infty)$ which has a composition series of length $n$.
    Since $n<\lambda$, we get by assumption of the induction that there exists a separable subalgebra $E_k'\subset A[\lambda,\infty)$ such that $D_k[\lambda,\infty)\subset E_k'$ and $\topdim(E_k'\cap A[\lambda,\lambda+n))\leq d$.
    Let $E_k\subset A$ be any separable $C^*$-algebra such that $D_k\subset E_k$ and $E_k[\lambda,\infty)=E_k'$.

    Case 3a: Assume $\lambda$ has cofinality at most $\omega$, i.e., there exist ordinals $0=\lambda_0<\lambda_1<\lambda_2<\ldots<\lambda$ such that $\lambda=\sup_k\lambda_k$.

    In this case, we let $\alpha_k:=\lambda_k$, and we set $D_0:=C$.
    Given $D_k$, we construct $E_k$ as described above.
    Given $E_k$, we get $D_{k+1}$ satisfying (3) by assumption of the induction.

    Case 3b:
    Assume $\lambda$ has cofinality larger than $\omega$.

    We start by setting $\alpha_0:=0$ and $D_0:=C$.
    Given $D_k$, we construct $E_k$ as described above.
    Given $E_k$, we define $\alpha_{k+1}$ as follows:
    \begin{align*}
        \alpha_{k+1} :=\inf \{\alpha \: |\: \alpha_k\leq\alpha\leq\lambda, \text{ and } E_k[0,\alpha) = E_k[0,\lambda)\}.
    \end{align*}
    Since $\lambda$ has cofinality larger than $\omega$ and $E_k$ is separable, we have $\alpha_{k+1}<\lambda$.
    Hence, we get $D_{k+1}$ satisfying (3) by assumption of the induction.

    From now on we treat the cases 3a and 3b together.
    Set $D:=\overline{\bigcup_kD_k}=\overline{\bigcup_kE_k}$.
    This is a separable sub-$C^*$-algebra of $A$ with $C\subset D$.
    Since $D[\lambda,\lambda+n) = \overline{\bigcup_k E_k[\lambda,\lambda+n)}$
    and $\topdim(E_k[\lambda,\lambda+n))\leq d$ for all $k$, we get $\topdim(D[\lambda,\lambda+n))\leq d$ from \autoref{HTprp:topdim:approx_GCR}.

    One checks that $D[0,\lambda) = \overline{\bigcup_k D_k[0,\alpha_k)}$.
    Since $\topdim(D_k[0,\alpha_k))\leq d$ for all $k$, we get $\topdim(D[0,\lambda))\leq d$, again by \autoref{HTprp:topdim:approx_GCR}.

    Then \autoref{HTprp:topdim:compSeries} gives:
    \begin{align*}
        \topdim(D[0,\lambda+n))
            =\max\{ \topdim(D[0,\lambda)), \topdim(D[\lambda,\lambda+n)) \}
            \leq d.
    \end{align*}
		
	This completes the proof.
\end{proof}

\begin{cor}
\label{HTprp:topdim:axiom6}
    The topological dimension of type I $C^*$-algebras satisfies the Marde\v{s}i\'c factorization axiom (D6), i.e., given a type I $C^*$-algebra $A$ and a separable sub-$C^*$-algebra $C\subset A$, there exists a separable $C^*$-algebra $D\subset A$ such that $C\subset D\subset A$ and $\topdim(D)\leq\topdim(A)$.
\end{cor}

\noindent
    This following theorem is the main result of this paper.
    It follows immediately from the above \autoref{HTprp:topdim:axiom6}, \autoref{HTprp:topdim:compSeries} and \autoref{HTprp:topdim:approx_GCR}.

\begin{thm}
\label{HTprp:topdim:dimThy}
    The topological dimension is a noncommutative dimension theory in the sense of \autoref{HTdfn:NCDimThy:ncDimThy} for the class of type I $C^*$-algebras.
\end{thm}

\begin{pgr}
\label{HTpgr:topdim:extension_to_all_Cs}
    Let us extend the topological dimension from the class of type I $C^*$-algebras to all $C^*$-algebras, as defined in \autoref{HTprp:NCDimThy:DimThy_extension_by_approx}.
    This dimension theory $\topdim\widetilde{}\;\colon\mathcal{C}^{*}\to\overline{\mathbb{N}}$ is Morita-invariant since $\topdim(A)=\topdim(A\otimes\mathbb{K})$ for any type I $C^*$-algebra $A$.

    If $\topdim\widetilde{}\;(A)<\infty$, then $A$ is in particular approximated by type I sub-$C^*$-algebras.
    This implies that $A$ is nuclear, satisfies the universal coefficient theorem (UCT), see \cite[Theorem~1.1]{Dad2003}, and is not properly infinite.
    It is possible that this dimension theory is connected to the decomposition rank and nuclear dimension, although the exact relation is not clear.

    Let us show that the (extended) topological dimension behaves well with respect to tensor products.
    First, if $A,B$ are separable, type I $C^*$-algebras, then $\Prim(A\otimes B)\cong\Prim(A)\times\Prim(B)$, see \cite[IV.3.4.25, p.390]{Bla2006}.
    This implies:
    \begin{align*}
        \topdim(A\otimes B)  &\leq\topdim(A)+\topdim(B).
    \end{align*}

    Next, assume $A,B$ are $C^*$-algebras with $\topdim\widetilde{}\;(A)=d_1<\infty$ and $\topdim\widetilde{}\;(B)=d_2<\infty$.
    This means that $A$ is approximated by separable, type I algebras $A_i\subset A$ with $\topdim(A_i)\leq d_1$, and similarly $B$ is approximated by separable, type I algebras $B_j\subset B$ with $\topdim(B_j)\leq d_2$.
    Then $A\otimes B$ is approximated by the algebras $A_i\otimes B_j$, and we have seen that $\topdim(A_i\otimes B_j)\leq d_1+d_2$.
    Thus:
    \begin{align*}
        \topdim\widetilde{}\;(A\otimes B)  &\leq \topdim\widetilde{}\;(A)+\topdim\widetilde{}\;(B).
    \end{align*}
    Note that we need not specify the tensor product, since $\topdim\widetilde{}\;(A)<\infty$ implies that $A$ is nuclear.
\end{pgr}

\section{Dimension theories of type I $C^*$-algebras}
\label{HTsec:typeI}

\noindent
    In this section we study the relation of the topological dimension of type I $C^*$-algebras to other dimension theories.
    It was shown by Brown, \cite[Theorem 3.10]{Bro2007}, how to compute the real and stable rank of a CCR algebra $A$ in terms of the topological dimension of certain canonical algebras $A_k$ associated to $A$.
    We use this to obtain a general estimate of the real and stable rank of a CCR algebra in terms of its topological dimension, see \autoref{HTprp:typeI:CCR-estimates}.
    Using the composition series of a type I $C^*$-algebra, we will obtain similar (but weaker) estimates for general type I $C^*$-algebras, see \autoref{HTprp:typeI:GCR-estimates}.
    \\

    Let $A$ be a $C^*$-algebra.
    We denote by $\rr(A)$ its real rank, see \cite{BroPed1991}, by $\sr(A)$ its stable rank, and by $\csr(A)$ its connected stable rank, see \cite[Definition 1.4, 4.7]{Rie1983}
    We denote by $A_k$ the successive quotient of $A$ that corresponds to the irreducible representations of dimension $k$.

    If $t$ is a real number, we denote by $\lfloor t\rfloor$ the largest integer $n\leq t$, and by $\lceil t\rceil$ the smallest integer $n\geq t$.

\begin{thm}[{Brown, \cite[Theorem 3.10]{Bro2007}}]
\label{HTprp:typeI:rr_sr_CCR}
    Let $A$ be a CCR algebra with $\topdim(A)<\infty$.
    Then:
    \begin{enumerate}[(1)   ]
        \item
        If $\topdim(A)\leq 1$, then $\sr(A)=1$.
        \item
        If $\topdim(A)>1$, then $\sr(A)=\sup_{k\geq 1} \max\{\left\lceil\frac{\topdim(A_k)+2k-1}{2k}\right\rceil,2\}$.
        \item
        If $\topdim(A)=0$, then $\rr(A)=0$.
        \item
        If $\topdim(A)>0$, then $\rr(A)=\sup_{k\geq 1} \max\{\left\lceil\frac{\topdim(A_k)}{2k-1}\right\rceil,1\}$.
    \end{enumerate}
\end{thm}

\noindent
    We may draw the following conclusion:

\begin{cor}
\label{HTprp:typeI:CCR-estimates}
    Let $A$ be a CCR algebra.
    Then:
    \begin{align}
        \label{eq:HTprp:typeI:CCR-estimates:1}
        \sr(A)  &\leq \left\lfloor\frac{\topdim(A)}{2}\right\rfloor+1, \\
        \label{eq:HTprp:typeI:CCR-estimates:2}
        \csr(A) &\leq \left\lfloor\frac{\topdim(A)+1}{2}\right\rfloor+1, \\
        \label{eq:HTprp:typeI:CCR-estimates:3}
        \rr(A)  &\leq \topdim(A).
    \end{align}
\end{cor}
\begin{proof}
    If $\topdim(A)=\infty$, then the statements hold.
    So we may assume $\topdim(A)<\infty$, whence we may apply \cite[Theorem 3.10]{Bro2007}, see \autoref{HTprp:typeI:rr_sr_CCR}.

    Let us show \eqref{eq:HTprp:typeI:CCR-estimates:1}.
    If $\topdim(A)\leq 1$, then $\sr(A)=1\leq\lfloor\topdim(A)/2\rfloor+1$ .
    If $d:=\topdim(A)\geq 2$, then we use $\topdim(A_k)\leq d$ to compute:
    \begin{align*}
        \sr(A) \leq \sup_k\max\{ \left\lceil \frac{d+2k-1}{2k} \right\rceil,2\}
        \leq \max\{ \left\lceil \frac{d+1}{2} \right\rceil,2\}
        \leq \left\lfloor \frac{d}{2} \right\rfloor+1.
    \end{align*}

    Now \eqref{eq:HTprp:typeI:CCR-estimates:2} follows from \eqref{eq:HTprp:typeI:CCR-estimates:1} since $\csr(A)\leq\sr(A\otimes C([0,1]))$ in general, by \cite[Lemma 2.4]{Nis1986}, and $\topdim(A\otimes C([0,1]))\leq\topdim(A)+1$, see \ref{HTpgr:topdim:extension_to_all_Cs}.

    To show \eqref{eq:HTprp:typeI:CCR-estimates:3}, we again use \cite[Theorem 3.10]{Bro2007}, see \autoref{HTprp:typeI:rr_sr_CCR}.
    If $\topdim(A)=0$, then $\rr(A)=0\leq\topdim(A)$ .
    If $d:=\topdim(A)\geq 1$, then  we use $\topdim(A_k)\leq d$ to compute:
    \begin{align*}
        \rr(A)
        \leq \sup_k\max\{\left\lceil \frac{d}{2k-1} \right\rceil,1\}
        \leq \max\{\left\lceil d \right\rceil,1\}
        \leq d,
    \end{align*}
    which completes the proof.
\end{proof}

\begin{rmk}
\label{HTrmk:typeI:transfinite_induction}
    What makes type I $C^*$-algebras so accessible is the presence of composition series with successive quotients that are easier to handle (i.e., of continuous trace or CCR), see \ref{HTpgr:prelim:typeI}.
    They allow us to prove statements by transfinite induction, for which one has to consider the case of a successor and limit ordinal.
    Let us see that for statements about dimension theories one only needs to consider successor ordinals.

    Let $(J_\alpha)_{\alpha\leq\mu}$ be a composition series, and $d$ a dimension theory.
    If $\alpha$ is a limit ordinal, then $J_\alpha=\overline{\bigcup_{\gamma<\alpha}J_\gamma}$, and we obtain:
    \begin{align*}
        d(J_\alpha)
        \leq_{(D5)}\; \sup_{\gamma<\alpha}d(J_\gamma)
        \leq_{(D1)}\; \sup_{\gamma<\alpha}d(J_\alpha),
    \end{align*}
    and thus $d(J_\alpha)=\sup_{\gamma<\alpha}d(J_\gamma)$.

    Thus, any reasonable estimate about dimension theories that holds for $\gamma<\alpha$ will also hold for $\alpha$.
    It follows that we only need to consider a successor ordinal $\alpha$, in which case $A=J_\alpha$ is an extension of $B=J_\alpha/J_{\alpha-1}$ by $I=J_{\alpha-1}$.
    By assumption the result is true for $I$ and has to be proved for $A$ (using that $B$ has continuous trace or is CCR).
    This idea is used to prove the next theorem.
\end{rmk}

\begin{thm}
\label{HTprp:typeI:GCR-estimates}
    Let $A$ be a type I $C^*$-algebra.
    Then:
    \begin{align}
        \label{eq:HTprp:typeI:GCR-estimates:1}
        \sr(A)      &\leq\left\lfloor\frac{\topdim(A)+1}{2}\right\rfloor+1, \\
        \label{eq:HTprp:typeI:GCR-estimates:3}
        \rr(A)      &\leq\topdim(A)+2.
    \end{align}
\end{thm}
\begin{proof}
    We will prove \eqref{eq:HTprp:typeI:GCR-estimates:1} by transfinite induction over the length $\mu$ of a composition series $(J_\alpha)_{\alpha\leq\mu}$ for $A$ with successive quotients that are CCR algebras.

    Set $d:=\topdim(A)$.
    Assume the statement holds for some ordinal $\mu$, and let us show it also holds for $\mu+1$.
    Consider the ideal $I:=J_\mu$ inside $A=J_{\mu+1}$.
    We obtain the following, where the first estimate follows from \cite[Theorem 4.11]{Rie1983}, and the second estimate follows by assumption of the induction for $I$ and \autoref{HTprp:typeI:CCR-estimates} for the CCR algebra $A/I$:
    \begin{align*}
        \sr(A)
        &\leq \max\{\sr(I),\sr(A/I),\csr(A/I)\} \\
        &\leq \max\{ \left\lfloor\frac{d+1}{2}\right\rfloor+1, \left\lfloor\frac{d}{2}\right\rfloor+1, \left\lfloor\frac{d+1}{2}\right\rfloor+1\} \\
        &=\left\lfloor\frac{d+1}{2}\right\rfloor+1.
    \end{align*}

    Let $\mu$ be a limit ordinal, and assume the statement holds for $\alpha<\mu$.
    This means that $\sr(J_\alpha)\leq\left\lfloor\frac{\topdim(J_\alpha)+1}{2}\right\rfloor+1$ for all $\alpha<\mu$.
    As explained in \autoref{HTrmk:typeI:transfinite_induction}, we obtain the desired estimate for $\mu$ as follows:
    \begin{align*}
        \sr(J_\mu)
        = \sup_{\alpha<\mu} \sr(J_\alpha)
        \leq \sup_{\alpha<\mu} \left\lfloor\frac{\topdim(J_\alpha)+1}{2}\right\rfloor+1
        = \left\lfloor\frac{\topdim(J_\mu)+1}{2}\right\rfloor+1.
    \end{align*}

    Finally, \eqref{eq:HTprp:typeI:GCR-estimates:3} follows from \eqref{eq:HTprp:typeI:GCR-estimates:1}, using the estimate $\rr(A)\leq 2\sr(A)-1$, which holds for all $C^*$-algebras, see \cite[Proposition 1.2]{BroPed1991}.
\end{proof}

\begin{rmk}
\label{HTrmk:typeI:rr_sr_GCR-estimate_sharp}
    It follows from \cite[Proposition 1.7]{Rie1983}, \autoref{HTprp:typeI:CCR-estimates}, and \autoref{HTprp:typeI:GCR-estimates} that we may estimate the stable rank of a $C^*$-algebra $A$ in terms of its topological dimension as follows:
    \begin{enumerate}[(1)   ]
        \item
        $\sr(A)=\left\lfloor\frac{\topdim(A)}{2}\right\rfloor+1$, \quad if $A$ is commutative.
        \item
        $\sr(A)\leq\left\lfloor\frac{\topdim(A)}{2}\right\rfloor+1$, \quad if $A$ is CCR.
        \item
        $\sr(A)\leq\left\lfloor\frac{\topdim(A)+1}{2}\right\rfloor+1$, \quad if $A$ is type I.
    \end{enumerate}

    This also shows that the inequality for the stable rank in \autoref{HTprp:typeI:CCR-estimates} cannot be improved (the same is true for the estimates of real rank and connected stable rank).

    To see that the estimate of \autoref{HTprp:typeI:GCR-estimates} for the stable rank cannot be improved either, consider the Toeplitz algebra $\mathcal{T}$.
    We have $\sr(\mathcal{T})=2$, while $\topdim(\mathcal{T})=1$.
\end{rmk}

\section*{Acknowledgments}

\noindent
I thank S{\o}ren Knudby for valuable comments and feedback.
I thank Mikael R{\o}rdam for interesting discussions and valuable comments, especially on the results in Section~3.
I also thank Wilhelm Winter for inspiring discussions on noncommutative dimension theory.


\providecommand{\bysame}{\leavevmode\hbox to3em{\hrulefill}\thinspace}

\end{document}